\documentclass[reqno]{amsart}

\usepackage[inner=1.0in,outer=1.0in,bottom=1.0in, top=1.0in]{geometry}
\usepackage{mathrsfs, amsmath,amssymb,amsfonts, amsthm}
\usepackage{color}
\usepackage{amscd}
\usepackage{graphics}
\usepackage{latexsym}

\newcommand{\ncal}{{\mathcal N}}

\newcommand{\Ss}{{\mathbb S}}

\newcommand{\PP}{{\mathbb P}}
\newcommand{\R}{{\mathbb R}}
\newcommand{\C}{{\mathbb C}}

\newcommand{\Z}{{\mathbb Z}}

\newcommand{\CP}{\C\PP}

\renewcommand{\phi}{\varphi}

\newcommand{\hcal}{\mathcal{H}}

\newcommand{\zcal}{\mathcal{Z}}

\newtheorem{theorem}{Theorem}[section]

\newtheorem{proposition}[theorem]{Proposition}
\newtheorem{lemma}[theorem]{Lemma}

\newtheorem{remark}[theorem]{Remark}
\newtheorem{definition}[theorem]{Definition}

\numberwithin{equation}{section}

\title{$2$-nodal domain theorems  for higher dimensional circle bundles}

\author{Junehyuk Jung}

\address{Department of Mathematics, Brown University, Providence, RI 02912 USA}

\email{junehyuk{\textunderscore}jung@brown.edu}

\author{Steve Zelditch}

\address{Department of Mathematics, Northwestern  University, Evanston, IL 60208, USA}

\email{zelditch@math.northwestern.edu}

\thanks{S. Z. was partially supported by NSF grant  DMS-1810747. J.J. was supported by NSF grant DMS-1900993, and by Sloan Research Fellowship.}

\begin{document}

\begin{abstract} We prove that the real parts of equivariant (but non-invariant) eigenfunctions of generic bundle metrics on nontrivial principal
$S^1$ bundles over manifolds of any  dimension have connected nodal sets and exactly 2 nodal domains. This generalizes earlier
results of the authors in the 3-dimensional case. The failure of the results on for  non-free $S^1$ actions is illustrated  on  even dimensional spheres
by one-parameter subgroups of rotations whose fixed point set consists of two antipodal points.
\end{abstract}

\maketitle

\section{Introduction}

In a recent article \cite{JZ20}, the authors proved a novel type of connectivity theorem for nodal sets of eigenfunctions
of certain Laplace operators on Riemannian  3-manifolds $(X, G)$ which are nontrivial principal circle bundles $\pi: X \to M$  over (real) surfaces $M$.
The
metrics $G$ are known as `bundle metrics', Kaluza-Klein metrics  or `connection metrics'
\footnote{They are also called  Sasakian metrics; see \cite{BBB} for an early study of them.}.
 They are induced by a choice of Riemannian metric $g$ on $M$ and a connection 1-form $\alpha$ on the nontrivial principal  circle bundle $\pi: X \to M$; see Lemma \ref{GLEM}. We denote the free circle $S^1$ action by $r_{\theta}$ and its infinitesimal generator by $Z = \frac{\partial}{\partial \theta}$. The connection determines an $S^1$ invariant splitting $T X = H \oplus V$ into horizontal
and vertical sub-bundles,  with $V = \R Z$, and $G$ is defined so that $H \perp V$, so that $G |_{H} = \pi^* g$ and so that $G |_{\pi^{-1}(z)}$ is the standard metric on $S^1 = \mathbb{R} /\Z$ along the fibers. The horizontal Laplacian $\Delta_H$  commutes with the vertical Laplacian $\frac{\partial^2}{\partial \theta^2}$, so that there exists an orthonormal basis  $\{\phi_{m, j}\}$ of joint eigenfunctions of $\Delta_G$ and  $\frac{\partial^2}{\partial \theta^2}$. We refer to such joint eigenfunctions
as ``equivariant eigenfunctions'' since they transform by $\phi_{m, j}(r_{\theta} x) = e^{i m \theta} \phi_{m, j}(x)$ under the $S^1$
action $r_{\theta}: S^1 \times X \to X$. They are called invariant when $m=0$.  It is proved  in \cite{JZ20} that the nodal sets of the real parts of non-invariant equivariant eigenfunctions are
connected for $m \not=0$ and that they divide $S^3$ into exactly two nodal domains if $0$ is a regular value of the eigenfunction; moreover, it is
a generic property of the bundle metrics that all eigenfunctions are  equivariant and that $0$ is a regular value of them all.   Until this result, sequences of
 eigenfunctions with a small number of nodal domains seemed to be rare, and it took ingenious arguments such as those of Lewy \cite{Lew77}  and Stern
\cite{CH2}  to construct them; but the result of \cite{JZ20} shows that the full orthonormal basis of eigenfunctions orthogonal to invariant functions
has this property for generic bundle metrics.
\subsection{Higher dimensional circle bundles}
The purpose of this Addendum is to generalize the  result to bundle Laplacians on  nontrivial principal circle bundles over manifolds of any dimension. The generalization requires two main  changes in the argument. First, our generic bundle metrics are of two kinds (i) general base metric perturbations, and (ii) perturbations of the connection $1$-form as in \cite{JZ20}. The perturbations (i) bring in new types of equations. Secondly, the key last step   that the single nodal
 component  divides $X$ into
precisely two nodal domains requires a very different argument from \cite{JZ20}. In addition, we illustrate what goes wrong with the proofs if the $S^1$
action is not free by illustrating the failure of the result on even-dimensional spheres.

Bundle metrics on  principal circle bundles can be defined in both a `top down' and a `bottoms up'  way. The top down approach is to start
with a principal $S^1$ bundle $\pi: X \to M$ and to construct a bundle metric $G$ on it from a metric $g$ on $M$ and a connection 1-form
$\alpha$ on the bundle  (see Section \ref{GEOMSECT}).
The $S^1$ action is generated by  a Killing vector field $Z = \frac{\partial}{\partial \theta}$, which commutes with any bundle Laplacian $\Delta_G$.
The norm $|Z|_G$ is called the lapse function; it is constant along fibers but  may  vary with the fiber. For  simplicity, when the $S^1$ action is
free,  we assume it is constant, in which case the fibers are geodesics. The bottoms up approach is to start with a suitable Riemannian manifold $(M, g)$ and a complex line bundle $L \to M$, to endow $L$ with a Hermitian metric
$h$ and to define $X$ as the unit bundle $\partial D^*_h \subset L^*$ of the Hermitian metric in the dual line bundle.  Given the circle bundle $X$
and a character of $S^1$, one may define the associated complex line bundle $L$ and Hermitian metric $h$ so that $X = \partial D^*_h $. Although the bottoms up and top down approaches are equivalent,
it is natural to separate the two approaches since the lead to different perturbation formulae.

\begin{definition}\label{EQEIGDEF}
Let $(X,G)$ be any Riemannian manifold with an $S^1$ action by isometries. Let $D_Z = \frac{1}{i} \frac{\partial}{\partial \theta} $ be the self-adjoint differential operator corresponding to the Killing vector field  $Z$ generating the $S^1$ action.  By an equivariant eigenfunction we mean
a joint eigenfunction
\[
\left\{ \begin{array}{l} \Delta_G \phi_{m, j} = - \lambda_{m, j}^2 \phi_{m, j}, \\
D_Z \phi_{m, j} = m \phi_{m, j}. \end{array} \right.
\]
\end{definition}

Equivariant eigenfunctions are complex-valued and their real and imaginary parts are (non-equivariant) $\Delta_G$-eigenfunctions. We are mainly interested
in the nodal sets of the real part,
\begin{equation} \label{NCALmjDEF}
\ncal_{m,j}: = \{\mathrm{Re} \phi_{m,j} =0\}.
\end{equation}
It is equivalent to study the nodal set of the imaginary part.

Our main result pertains to nodal sets \eqref{NCALmjDEF} for Laplacians of  generic bundle metrics. To define `generic',
we specify a class of metrics  $g$ on $M$  (Section \ref{GENMET})  and of connection $1$-forms $ \alpha$ (Section \ref{VARYALPHA})  in some Banach or Frechet space, and `generic' will refer to a residual
subset of that space. Either we fix the base metric $g$ and vary the connection, or we fix the connection and allow general variations
of base metrics.

\begin{theorem} \label{MAIN}
For  generic data $(g, \alpha)$ of a bundle (Kaluza-Klein)  metric $G$ on a nontrivial principal circle bundle $\pi : X \to M$,
\begin{itemize}
\item $0$ is a regular value of all eigenfunctions of the bundle Laplacian $\Delta_G$.
\item All eigenvalues are simple. The eigenfunctions are joint eigenfunctions of $\frac{\partial^2}{\partial \theta^2}$ and of $\Delta_G$.
\item Except for pullback (invariant) eigenfunctions, the nodal set \eqref{NCALmjDEF}  is connected and there are exactly two nodal domains.
\end{itemize}
\end{theorem}
The nodal domains of pullback eigenfunctions are easily seen to be the inverse image of the nodal domains in $M$, so the above theorem
is best possible.

The proof of Theorem \ref{MAIN} splits up into two distinct parts. The first part is to prove that eigenfunctions of generic bundle
metrics of the two basic types have the stated properties. The second part is to prove that zero sets \eqref{NCALmjDEF} of real parts of equivariant eigenfunctions $\phi_m$
on $X$ have the stated connectivity properties if $0$ is a regular value.

 Equivariant functions  $\phi_{m, j}$ on the circle bundle transforming by $e^{i m \theta}$   correspond to sections of $L^m$.
 There is a canonical lift map $s \to \hat{s}$ taking sections of $L^m$ to equivariant functions on $X$. Theorem \ref{MAIN} can therefore be restated
 in terms   of  real and imaginary parts of equivariant
eigensections $s_{m,j}$  of $L^m$,
\begin{equation} \label{REALHYP}
Z_{\mathbb{R} \hat{s}_{m,j} } : = \{\mathrm{Re} \hat{s}_{m,j} (\zeta) = 0, \zeta \in \partial D_h^*\}.
\end{equation}
See Section \ref{EQEIGSECT} for background. It is often easier to calculate with the associated eigensections, since connections on $L^m$ are simpler
to work with than horizontal derivatives.

\subsection{Non-free $S^1$ actions on even dimensional spheres}

Odd dimensional spheres fit into the framework of bundle metrics on circle bundles over K\"ahler manifolds (namely
$\Ss^{2m -1} \to \CP^m$), although the metrics are very non-generic:  eigenvalues have high multiplicities, $0$ is not a regular value
of many eigenfunctions (spherical harmonics) of a given eigenvalue (or, degree), and the there may exist many nodal domains
in the singular case. On the other hand, as in \cite{JZ21}, the real part of a  random equivariant spherical harmonic on $\Ss^{2m-1}$ does have a connected
nodal set with exactly two nodal domains.

 If we drop the   assumption that $X \to M$ is a  principal circle bundle and allow
the $S^1$ action to have fixed points, then most of Theorem \ref{MAIN} fails. The simplest example
where  the conclusion of Theorem \ref{MAIN} is false is  that of the $S^1$ action by rotations around
an axis  for the standard metric on $S^2$. The equivariant eigenfunctions
are the usual spherical harmonics $Y_N^m$. The nodal sets of their real parts are connected but they have $m N$ nodal domains. Moreover,
$0$ is not a regular value and they have $m N$ singular points. We show that much of this is true for similar $S^1$ actions on spheres of any even dimension.
The following is proved in Section \ref{EVENDIM}.

\begin{lemma}\label{BADSPHERE}  Suppose  that $X= \Ss^{2n}$ be an even dimensional sphere and $S^1$ is a 1-parameter subgroup of $SO(2n +1)$ with exactly two fixed
points. Then,
\begin{itemize}
\item for any $S^1$-invariant metric $G$, the nodal sets of the real parts of the  joint eigenfunctions of the $S^1$ action and $\Delta_G$ are connected.
\item $0$ is never a regular value for real parts of  the  joint eigenfunctions of the $S^1$ action and $\Delta_G$.
\item The set of singular points (critical points on the nodal set) of real parts of  the equivariant eigenfunctions is
exactly of co-dimension 2.
\end{itemize}
\end{lemma}

\subsection{Comments}

To avoid duplicating material in  \cite{JZ20}, we only give a detailed presentation of the new steps in the
generalization and refer to \cite{JZ20} for much of the background and references. However, for the sake of readability, we do state
some overlapping  calculations of Laplacians and their perturbations of connection 1-forms  from \cite{JZ20}. As mentioned above, the generic metric data in higher dimensions is quite different from that in dimension two, and the calculations are quite different.

We refer the readers to \cite{CJLS}, published shortly after the release of our current article. While there are overlapping results, it is important to note that the two works were conducted independently, and the proof methods used are different.

\section{Geometric background}\label{GEOMSECT}

In this section we introduce the geometric data that goes into the
construction of bundle (i.e connection or Kaluza-Klein) metrics on a circle bundle, and give formulae for the associated  Bochner connection
Laplacians $\nabla^* \nabla$ and Kaluza-Klein Laplacians $\Delta_G$ associated to different types of bundle Laplacians.

\subsection{Classification of $S^1$ bundles and bundle metrics}
In this section we consider the general top-down approach to constructing bundle metrics for principal $S^1$ bundles $\pi: X \to M$.
Let $M$ be a compact manifold. Then there is a 1-1 correspondence between equivalence classes of circle bundles over $M$ and elements of
$H^2(M, \Z)$. Given an integral closed 2-form $\omega$ on $M$ there is a circle bundle $\pi: X \to M$ with connection form $\alpha$
such that $\pi^*\omega = d \alpha$. The cohomology class   $c_1 = [\omega] \in H^2(M, \Z)$ of the associated complex line bundle is its Chern class.

A connection 1-form $\alpha$ is an $S^1$-invariant  1-form satisfying $\alpha(Z) =1$. The difference $\alpha_1 - \alpha_2$ of two connections is an $S^1$ invariant 1-form which annihilates $Z$ and is therefore horizontal. Hence it descends to $M$ as a real-valued one form. It follows that the space of connection
$1$-forms on a fixed circle bundle is parameterized by $\Omega^1(M,\mathbb{R})$.

The curvature form of the $S^1$ bundle is the $S^1$ horizontal form  defined by $\omega = d \alpha$. It descends to $M$ as a closed 2-form.
  When $\omega$
is a symplectic form, the connection is called `fat' \cite{W80} and then $(M, \omega)$ is a symplectic manifold; in this case $X$ is odd-dimensional.

 Let
$TX = VX \oplus HX$ where $HX$ is the horizontal space for $\alpha$. A Riemannian metric  $g$ together
with $\alpha$ determines the bundle metric as long as  $Z = \frac{\partial}{\partial \theta}$ is  a geodesic. More generally, we can allow the  {\it lapse function }
$N(Z) = G(Z, Z)$ to be non-constant. We do not pursue this generalization but summarize the statements in the general setting.


\begin{lemma}  \label{GLEM} A bundle metric on a principal circle bundle $\pi: X \to M$  is defined by the data $(N, \alpha, g)$ with  $$G(N, \alpha, g)  := N^2 \alpha \otimes \alpha +\pi^* g$$ where $\alpha$ is the  Killing invariant connection 1-form
 with $\alpha(Z) =1$. The $S^1$ orbits are geodesics if and only if $N \equiv C$ for some $C > 0$.
\end{lemma}
The volume form of the metric is given by,
$$dV_G  = \alpha \wedge \pi^* dV_g. $$
We often assume $N \equiv 1$ and then write $G = G(\alpha, g)$.

\subsection{\label{ASSOCIATED} Associated line bundles and Hermitian metrics }\label{LBSECT}

We assume that  $X$ is a principal $S^1$ bundle $\pi: X \to M$ with action $r_{\theta}: X \to X$ for $e^{i \theta} \in S^1$. Given any character $\chi_k(e^{i \theta}) = e^{i k \theta}$
one can form the associated complex line bundles $L^k : = X \times_{\chi_k} \C \to M$.

Conversely, given a complex line bundle $L \to M$, circle bundles $X$ can be constructed using Hermitian metrics $h$ on $L$. Given $h$ we define
$$X_h : = \partial D_h^* =
\{(z, \lambda) \in L^*: |\lambda|_z = 1 \} $$
be the boundary of the unit co-disk bundle in the dual line bundle $L^*$.

 We  fix a connection,
$$\nabla: C^{\infty}(X, L) \to C^{\infty}(X, L \otimes T^*X). $$
Given a connection $\nabla$ on $L$ and a vector field $V$ on $X$,
the covariant derivative of a section $s$ is defined by $\nabla_V s =
\langle \nabla s, V \rangle$. The curvature is the 2-form $\Omega^{\nabla}$
defined by $\Omega^{\nabla}(V, W) = \lbrack\nabla_V, \nabla_W\rbrack - \nabla_{\lbrack V,W\rbrack }.$
If $e_L$ is a local frame and $\nabla e_L = \alpha \otimes e_L$ then
$\Omega^{\nabla} = d \alpha$.

The value in introducing the line bundle is that $S^1$-equivariant functions transforming by $\chi_k$ on $X$ correspond to sections of $L^k \to M$. More precisely,
 there is a natural lift-map isomorphism,
$$s_k \in L^2(M, L^k) \to \hat{s} \in L^2_k(\partial D_h^*), \;\; \hat{s}_k(z, \lambda) : = \lambda^{\otimes k} (s(z)) $$
from  sections of the kth tensor power $L^k$ of $L$  to equivariant  functions on $\partial D_h^*$. The natural $\C^*$ action on $L^*$ restricts to
an $S^1$ action on $\partial D_h^*$ and $\hat{s}_k(r_{\theta} x) = e^{i k \theta} \hat{s}_k(x). $ It is evident that the zero set of $\hat{s}_k$ is
a circle bundle over the zero set of $s_k$. However, we are most interested in the zero set of $\mathrm{Re} \hat{s}_k$.

\section{Laplacians}

Let $G =  G(g, \alpha)$ be a bundle metric on  a circle bundle $\pi: X \to M$ as in Lemma \ref{GLEM}. The bundle Laplacian $\Delta_G$ has the decomposition
\[
\Delta_G = \Delta_H + \Delta_V, \;\; \Delta_V =  \frac{\partial^2}{\partial \theta^2},\;\;\;
\]
where $\Delta_V$ is the vertical Laplacian and where $\Delta_H$
is the horizontal Laplacian. The equation $\Delta_V =  \frac{\partial^2}{\partial \theta^2}$ holds because we assume that lapse function $N$ equals $1$.
The weight spaces
are $\Delta_H$-invariant, i.e. $\Delta_H: \hcal_m \to \hcal_m$.

\subsection{ Equivariant eigenfunctions and eigensections}\label{EQEIGSECT}
Under
the  canonical identification using the lifting map of Section \ref{EQEIGSECT}, $\Delta_H |_{\hcal_m}$ restricts to the Bochner connection Laplacians $D_m =  \nabla_m^* \nabla_m $ on sections
of $L^m$ in the sense that
$\hcal_m \cong L^2(X, L^m), \;\;\;\Delta_H |_{\hcal_m} \cong D_m - m^2 I.$
In the bottoms-up approach where $X$ is defined as the unit bundle in $L^*$ relative to a hermitian metric $h$,
 Bochner Laplacians are defined by $\nabla^*_{h,g} \nabla$ on $L^2(X, L)$
equipped with the data $(g, h, J, J_L, \nabla)$, where  $g$ is a metric
on $X$, and $\nabla$ is a connection on $L$.
If $e_L$ is a local frame for $L$, then
$$ \widehat{\nabla_m^* \nabla_m (f e_L^m)} =
\Delta_H   \widehat{(f e_L^m)}.$$

An equivariant  eigenfunction  $\phi_{m,j}$ as in Definition \ref{EQEIGDEF} corresponds under the natural lifting map to an equivariant eigensection $f_{m,j} e_L^m$
of $L^m$ in a local frame $e_L^m$.
Let
\begin{equation*}
\mathrm{Re} f_{m,j} = a_{m,j}(z), \;\; \mathrm{Im} f_{m,j} = b_{m,j}(z).
\end{equation*}
Then,
\[
f_{m,j}(z) e^{-i m\theta} = (a_{m,j}(z) + i  b_{m,j}(z))(\cos m \theta - i \sin m \theta),
\]
so that with $\phi_{m, j} = u_{m,j} + i v_{m,j}$,
 \begin{equation} \label{ubab} \left\{ \begin{array}{l}
 u_{m,j} =  a_{m,j} \cos m \theta  +  b_{m,j} \sin m\theta,\\
v_{m,j} = b_{m,j} \cos m \theta  -a_{m,j} \sin m \theta.
\end{array} \right. \end{equation}

 \subsection{ Hilbert spaces of sections}

In the top-down approach, the  Hilbert space inner product on $L^2(X, dV_G)$ is with respect to the Riemannian volume form $dV_G$.
In the bottoms up approach,  we let $(L,h) \to M$ be a Hermitian  line bundle (which we assume to be holomorphic in the K\"ahler setting). We thus have a pair
of metrics, $h$ resp. $g$ (with K\"ahler form $\omega_{\phi}$) on $L$ resp. $TM$.
To each pair $(h,g)$ of metrics we associate Hilbert space inner products $\mathrm{Hilb}_m(h,g)$ on
sections
$s \in L^2_{m \phi}(X, L^m)$ of the form
$$\|s\|_{h^m}^2 : = \int_X |s(z)|^2_{h^m} dV_{g}, $$
where $ |s(z)|^2_{h^m}$ is the pointwise Hermitian norm-squared of
the section $s$ in the metric $h^m$. In a local  frame $e_L$, we write
\[
\|e_L\|_h^2 = e^{- \psi}.
\]
 In local coordinates $z$ and the local frame $e_L^m$ of $L^m$, we may write $s = f e_L^m$ and then
\[
|s(z)|^2_{h^m} = |f(z)|^2 e^{- m \psi(z)},
\]
and
\[
\| f e_L^m\|_{h^m}^2 : = \int_X |f(z)|^2e^{-m \psi} dV_{g}.
\]

\subsection{\label{QUADSECT} Quadratic forms} The horizontal Laplacian is the Laplacian $\Delta_H = d_H^* d_H$ where $d_H = \pi_H d$ is the horizontal part of
the exterior derivative and where $d_H^*$ is the adjoint with respect to the bundle metric $G$. Thus, for $F \in C^{\infty} (X)$,
$\langle \Delta_H F, F \rangle = \int_X |d_H F|^2_{G} d V_G. $

We now trivialize the bundle $\pi: X \to M$ by choosing a local {\it unitary} section ${\mathfrak u}: U \subset M \to X$ (which may be taken to be global on a set of full measure).
We then write $x = r_{\theta} {\mathfrak u}(y), $ and use $(y, \theta)$ as local coordinates on $X$. In these coordinates $F(y, \theta) = f(y) e^{i m \theta}$ for
the induced local function $f(y) = {\mathfrak u}^* F$ on $U$.  From a global viewpoint, $f$ is a section of $L^m$. Then $df = {\mathfrak u}^* d F$ and $d F = (d f) e^{i m \theta} + i m f e^{i m \theta} d \theta$. Hence,
$$  |d F - i m F \alpha|^2_{G} = | (d f) e^{i m \theta} + i m f e^{i m \theta} (d \theta - \alpha)|^2 = | d f   + i m f  (d \theta - \alpha)|^2 $$
Further, we define the local connection 1-form  ${\mathfrak u}^* (\alpha - d \theta) $ in the unitary frame ${\mathfrak u}$. Note that
 $(d \theta - \alpha)(Z) = 0$, so $ (\alpha - d \theta)$  is a horizontal  $1$-form.  We claim that on $\pi^{-1}(U), \alpha =  d \theta +\pi^* {\mathfrak u}^* (\alpha - d \theta). $  We
 first check the equation at the points of ${\mathfrak u}(M)$. Since  ${\mathfrak u} \circ \pi = \mathrm{Id} : {\mathfrak u}(M)  \to {\mathfrak u}(M)$,  $\alpha _{ {\mathfrak u}(y)}= (d \theta +\pi^* {\mathfrak u}^* (\alpha - d \theta)) |_{u(y)} $.
To check it at a general point $r_{\theta} u(y)$ we note that both $d \theta$ and $\pi^* {\mathfrak u}^* (\alpha - d \theta)$ are $S^1$-invariant.

Before stating the next result, we summarize the notation and terminology.
\begin{definition}
Given a bundle metric $G = G(g, \alpha)$ and its Laplacian $\Delta_G$,
\begin{itemize}
\item An equivariant eigenfunction of $\Delta_G$ of equivariant degree (or, weight) $m$ is a joint eigenfunction $\phi_{m, j}: X \to \C$ transforming
by $\phi_{m, j}(r_{\theta}(x)) = e^{i m \theta} \phi_{m, j}(x)$.
\item In a unitary frame ${\mathfrak u}: U \subset M \to X$, and the associated trivialization $X | _U = {\mathfrak u}(M) \times S^1$, we say that the function $f_{m,j} : U \to \C$
defined by $f_{m,j} = {\mathfrak u}^* \phi_{m, j}$  is the local expression of the   equivariant eigenfunction in the frame $ {\mathfrak u}$. If we change the
frame from ${\mathfrak u}$ to a local holomorphic frame $e_L^m$, then we obtain a local expression for the holomorphic section of $L^m$ whose lift is
  $\phi_{m,j}$.  \end{itemize}
\end{definition}

\begin{lemma}\label{LLEM}  Let $G = G(g, \alpha)$ be a bundle metric for a circle bundle $\pi: X \to M$ as in Lemma \ref{GLEM}. Let  $F \in H_m(X)$
and as above express $F(y, \theta) = f(y) e^{i m \theta}$  in the local unitary frame $u: U \to X$ and local trivialization $X|_U \simeq U \times S^1$. Then,
the quadratic form   $Q_m(F) = \langle \Delta_H F, F \rangle $ on $F \in L^2_m(X, dV_G)$  equals $Q_m(F) =    \langle L_m f, f \rangle_{L^2(M)},$ where
\begin{equation} \label{LmDEF}
L_m f  =-\Delta_g f + m^2  |  {\mathfrak u}^* (\alpha - d \theta)|^2_{g} f -  i m  d_g^* ( {\mathfrak u}^* (\alpha - d \theta))f.
\end{equation}
\end{lemma}
\begin{proof}

Note that for any covector $\xi \in T_x^* X$, $\pi_H(\xi) = \xi - \langle \xi, Z^{\flat} \rangle Z^{\flat}  = \xi - \xi(Z) Z^{\flat} \in H_x(X), $ where $Z^{\flat} $ is the 1-form dual to $Z$ in the sense that $Z^{\flat}$ is vertical, $Z^{\flat}(Z) = 1$ and $Z^{\flat} |_{H(X)} = 0$. Clearly, $Z^{\flat} =\alpha$ so that
$$\pi_H(\xi) =  \xi - \xi(Z) \alpha. $$ Hence,
$$\langle \Delta_H F, F \rangle = \int_X |d F - dF(Z) \alpha|^2_{G} d V_G. $$
If $F \in H_m(X)$ then $dF(Z) = i m F$ and we get
$$\begin{array}{lll} \langle \Delta_H F, F \rangle & = &  \int_X |d F - i m F \alpha|^2_{G} d V_G\\&&\\
& =& \int_M \left(\int_{\pi^{-1}(y)}   |d F - i m F \alpha|^2_{G}d\theta  \right)  dV_g(y) . \end{array} $$

It follows that,
$  |d F - i m F \alpha|^2_{G}  = | d f   - i m f  {\mathfrak u}^*(\alpha - d \theta)  |^2 $ and
\begin{align*}
 \langle \Delta_H F, F \rangle & =    \int_M   |d f - i m f  \mathfrak{u}^* (\alpha - d \theta)|^2_{g}   dV_g(y) \\
 & =  \int_M \left( |df|^2_g  +  im    \langle df,f{\mathfrak u}^* (\alpha - d \theta)\rangle -im \langle f{\mathfrak u}^* (\alpha - d \theta),df  \rangle + m^2 |f|^2 |  {\mathfrak u}^* (\alpha - d \theta)|^2_{g} \right) dV_g\\
 & =  \langle L_m f, f \rangle_{L^2(M)},
\end{align*}
where
\begin{align*}
L_m f  &= -\Delta_g f + m^2  |  {\mathfrak u}^* (\alpha - d \theta)|^2_{g} f + i m \langle df,  {\mathfrak u}^* (\alpha - d \theta) \rangle
-  i m  d_g^* (f {\mathfrak u}^* (\alpha - d \theta))\\
&=-\Delta_g f + m^2  |  {\mathfrak u}^* (\alpha - d \theta)|^2_{g} f -  i m  d_g^* ( {\mathfrak u}^* (\alpha - d \theta))f.
\end{align*}
In the last equation, we used
\[
d_g^*(\phi w)  = g^*(d\phi, w) + \phi d_g^* w
\]
for any $0$-form $\phi$ and any $1$-form $w$.
\end{proof}

\section{ Perturbation formulae}\label{PERTURBATIONSECT}
In this section, we prepare for the proof of the first two statements of Theorem \ref{MAIN} in the next section (See Theorem \ref{GENERIC1})
by calculating variations of the relevant Laplacians and their associated quadratic forms. The main Laplacian is $\Delta_G$, but only
its horizontal part $\Delta_H$ varies as the data $(g, \alpha)$ is varied and so we concentrate on that. In turn, $\Delta_H = \bigoplus_m L_m$
where $L_m$ is defined in  \eqref{LmDEF} , hence it suffices to calculate variations of $L_m$ under variations of the base metric or the connection.
It is simpler to calculate variations of operators on sections of $L^m$ than their lifts to $X$.

 Under any variation (i.e. along a curve of data), we have by \eqref{LmDEF},
\begin{equation} \label{LMDOT}
\dot{L}_m f  = - \dot{\Delta}_g f + m^2  \dot{g}^*  \langle  {\mathfrak u}^* (\alpha - d \theta),   {\mathfrak u}^* (\alpha - d \theta)\rangle  f
-  i m \dot{d}_g^* (  {\mathfrak u}^* (\alpha - d \theta))f.
\end{equation}
Here we use the notation $\dot{F} = \frac{d}{dt} |_{t=0} F_t$ for the infinitesimal deformation (or, variation) of $F$ along a curve.

\begin{remark}  It is sometimes
notationally inconvenient to use the ``dot'' notation and we also use the notation $\delta F$ for $\dot{F}.$
\end{remark}

We now calculate how the various terms deform when $g$ or $\alpha$ is deformed.  When applying the results to prove generic properties,
we need to move the derivatives on the metric or connection data onto  the other factors; hence it is advantageous to evaluate the variation of the quadratic form
associated to \eqref{LMDOT} rather than the operator $\dot{L}_m$ with derivatives applied to factors not being varied.

\subsection{\label{GENMET} General metric perturbations}

\begin{lemma}\label{JJLEM1}
Let $g(t)$ be a $1$-parameter deformation of $g$ with
$g(0) = g$ and $b(t) g^{-1} = g(t)^{-1}$. Then for $f \in C^\infty(M)$ and $W \in L^1(M)$, we have
\[
\langle \dot{L}_m f, W\rangle_{L^2(g)} = \int_M (df + imf\eta, dW+imW\eta)_{bg^*} dV_g + \frac{1}{2}\int_M (df + imf\eta, W d \mathrm{Tr}(b) )_{g^*}dV_g,
\]
where $b=b'(0)$.
\end{lemma}
\begin{proof}
Denote by $L_{m}(t)$ the horizontal Laplacian corresponding to $g(t)$ with $L_m(0) = L_m$. Let $\eta = u^*( \alpha-d\theta)$. For any $f \in C^\infty(M)$
and $W \in L^1(M)$, we have
\begin{align*}
\langle L_m(t) f, W\rangle_{L^2(g)} &= \langle L_m(t) f, W \sqrt{\det b(t)}\rangle_{L^2(g(t))}\\
&= \int_M (df + imf\eta, d(W \sqrt{\det b(t)}) + im W \sqrt{\det b(t)}\eta)_{g^*(t)}dV_{g(t)}\\
&= \int_M (df + imf\eta, dW+imW\eta)_{g^*(t)} dV_g + \int_M (df + imf\eta, W d\sqrt{\det b(t)})_{g^*(t)}dV_{g(t)}.
\end{align*}
Here $dW$ is the distribution derivative of $W\in L^1$. The second integral is
\begin{align*}
\int_M (df + imf\eta, W d\sqrt{\det b(t)})_{g^*(t)}dV_{g(t)} &= \int_M (df + imf\eta, W \frac{d\sqrt{\det b(t)}}{\sqrt{\det b(t)}})_{g^*(t)}dV_g\\
&=\int_M (df + imf\eta, W d \log \sqrt{\det b(t)})_{g^*(t)}dV_g\\
&= \frac{1}{2}\int_M (df + imf\eta, W d \log \det b(t))_{g^*(t)}dV_g.
\end{align*}
Therefore
\begin{align*}
&\langle \dot{L}_m f, W\rangle_{L^2(g)}\\
=& \frac{d}{dt}\langle L_m(t) f, W\rangle_{L^2(g)}\bigg|_{t=0}\\
=&\int_M (df + imf\eta, dW+imW\eta)_{bg^*} dV_g + \frac{1}{2}\int_M (df + imf\eta, W d \log \det b(0))_{bg^*}dV_g\\
&+ \frac{1}{2}\int_M (df + imf\eta, W d \left(\frac{d}{dt}\log \det b(t)\bigg|_{t=0}\right))_{g^*}dV_g\\
=&\int_M (df + imf\eta, dW+imW\eta)_{bg^*} dV_g + \frac{1}{2}\int_M (df + imf\eta, W d \left(\frac{d}{dt} \det b(t)\bigg|_{t=0}\right))_{g^*}dV_g,
\end{align*}
where we used $\det b(0) = 1$ to eliminate the second term and simplify the third term. Now because $b(0)$ is the identity matrix, we see that
\[
\frac{d}{dt} \det b(t)\bigg|_{t=0} = \mathrm{Tr}(b),
\]
and so we have
\[
\langle \dot{L}_m f, W\rangle_{L^2(g)} = \int_M (df + imf\eta, dW+imW\eta)_{bg^*} dV_g + \frac{1}{2}\int_M (df + imf\eta, W d \mathrm{Tr}(b) )_{g^*}dV_g. \qedhere
\]
\end{proof}

\subsubsection{ Varying the connection}\label{VARYALPHA}
We now vary $\alpha$ or equivalently its local expression $u^* (\alpha - d \theta)$ in the
unitary frame $u$. The deformation equation is essentially the same as in \cite{JZ20}.
\begin{lemma}
The variation $\dot{L}_m$ under a variation of the connection 1-form $\alpha$ is given by, $$\dot{L}_m f =   \left(   - 2 i  G(df, \dot{\alpha})   + if d_g^* \dot{\alpha} + 2G(\dot{\alpha}, \alpha)  f\right) e_L.$$
\end{lemma}

This calculation is precisely the same as in \cite{JZ20}, so only briefly review it.
  $$  \nabla^* \nabla (f e_L)  = \left(- \Delta_g f   - 2 i  g^*(df, \alpha)   + if d_g^*\alpha + g^*(\alpha, \alpha)  f\right) e_L,$$
  where $\Delta_g f$ is the scalar Laplace operator. Taking the variation
  with respect to $\alpha$ gives,
    $$\delta  \nabla^* \nabla (f e_L)  = \left(   - 2 i  G(df, \dot{\alpha})   + if d_g^* \dot{\alpha} + 2G(\dot{\alpha}, \alpha)  f\right) e_L.$$

   \section{Generic properties of eigenvalues and eigenfunctions of general prinicipal $S^1$ bundle Laplacians}
In this section we prove the first two statements of Theorem \ref{MAIN}, which we reformulate as Theorem \ref{GENERIC1}.  We prove the third statement about nodal sets in the next section.

\begin{theorem}\label{GENERIC1}

For generic bundle metrics $G = G(g, \alpha)$ with $g, \alpha$ generic data  in one of the ways listed below, all of the eigenvalues of $\Delta_G$ are
simple and all of the equivariant eigenfunctions have  $0$ as a regular value.
Equivalently  for every $m$,  the spectrum of
the operator $L_m$ \eqref{LmDEF}
 is simple, the spectra of  $L_m$ and $L_{m'}$ are disjoint if $m \not= m'$,    and  all of its eigensections have zero as a regular value.
 Moreover, if we lift sections to equivariant eigenfunctions $\phi$,
then $\mathrm{Re} \phi$ and $\mathrm{Im} \phi$ have zero as a regular value.

 The generic admissible data is of the following kinds:

  \begin{itemize}
  \item[(i)] We fix $\alpha$ and vary $g$ among all metrics on $M$;

  \item[(ii)] We fix $g$ and vary $\alpha$ among connection 1-forms on $\pi: X \to M$.

   \end{itemize}

   \end{theorem}

As in \cite{JZ20} we prove Theorem \ref{GENERIC1} using the approach of Uhlenbeck \cite{U}. We review this method before giving the proof.

\subsection{Review of Uhlenbeck's transversality and genericity results}

K.  Uhlenbeck has given a general set-up for proving that the spectra of generic elliptic operators are simple and that the eigenfunctions
have $0$ as a regular value \cite{U}.  The proofs are based on the Thom-Smale infinite dimensional transversality theorem rather than on Kato-Rellich
perturbation theory as in \cite{BU}.We adapted her approach to  $S^1$ bundle Laplacians over surfaces in \cite{JZ20}. For the sake of completeness,
we briefly review the
 proof  that for generic metrics on compact $C^r$Riemannian manifolds,
 all eigenvalues are simple and  all eigenfunctions have $0$ as a regular value.

 One considers a family $L_b$  of elliptic operators depending on data $b $ in a Banach space $B$. One mainly needs that $B$ is a Baire space,
 i.e. that residual subsets are dense. As in  \cite{JZ20}, we define $B$ to be a $C^k$ space of data $(g, \alpha)$  needed to define a bundle metric
 as in Lemma \ref{GLEM}. The candidate eigenfunctions
 are denoted by $u$.

As in \cite[Section 2]{U}, we study the map $\Phi: C^k(X) \times \mathbb{R} \times B \to \mathbb{R}$ given by
\begin{equation} \label{CAPPHI} \Phi(u, \lambda, b) = (L_b + \lambda) u. \end{equation}
The zeros of the map correspond to Laplace type operators $L_b$ together with their eigenvalues $\lambda$ and eigenfunctions $u$. This map is sufficient
to prove generic simplicity of eigenvalues. To prove the more difficult fact that $0$ is a regular value of all eigenfunctions, Uhlenbeck introduces the additional
maps,

\begin{itemize}

\item $Q := \{(u, \lambda, g) \in C^k(X) \times \mathbb{R}_+ \times B: \Phi(u, \lambda, b) = 0\}$.
\bigskip

\item $\alpha: Q \times M \to \C: \alpha(u, \lambda, b, x) = u(x). $\bigskip


\end{itemize}

We use the following `abstract genericity'  result of \cite[Theorem 1]{U}- \cite[Lemmas 2.7-2.8]{U}.

\begin{theorem} \label{U} Assume that $\Phi$ \eqref{CAPPHI}  is $C^k$ and has zero as a regular value.  Then the eigenspaces
of $L_b$ are one-dimensional.
If additionally, $\alpha: Q \times M \to \mathbb{C}$ has zero as a regular value,  then additionally
\[
\{b \in B: \text{ the eigenfunctions of } L_b\; \text{ have } 0\text{as a  regular  value} \}\; \text{ is  residual in} B.
\]
\end{theorem}

The next two   propositions give a tool to verify the first  hypothesis of Theorem \ref{U}.
  (see  \cite[Proposition 2.4 and Proposition 2.10]{U}). \begin{proposition}\label{PROP} We have,
\begin{enumerate}
  \item If at points of $\Phi^{-1}(0)$, $D_2 \Phi : T_bB \to H_{k-2}^p(X) \subset H_{-1}^q(X)$ has dense image in $H^q_{-1}(M),  1 < q < \infty$, then
  $0$ is a regular value of $\Phi$;.
  \item Let $J = \mathrm{im} D_2 \Phi$ and assume that for $W \in L^1(M)$ and
$W \in C^2(M - \{y\})$, the property
\[
\int_M W(x) j(x) d\mu_x = 0
\]
for all $j \in J$ implies $W = 0$. Then  $\Phi$ is $C^k$ and has zero as a regular value.
\end{enumerate} \end{proposition}

The second statement follows from  \cite[Lemma 2.7]{U}:  Let
 $\pi: Q \to B$ be a $C^k$ Fredholm map of index $0$. Then if
$f: Q \times X \to Y$ is a $C^k$ map for $k$ sufficiently large and if $f$
is transverse to $Y'$ then $\{b \in B: f_b: = f|_{\pi^{-1}(b)} \; {\mathrm is\; transverse\; to\;} Y'\}$ is residual in $B$.
Let $$\alpha: f^{-1}(Y') \to B \; {\mathrm be}\; \alpha: f^{-1}(Y') \subset Q \to B.$$

\begin{lemma} \label{REG0LEM} The eigenfunctions of $L_b$ have zero as a regular value
if  $b$ is a regular value of $\pi$ and if $0$ is a regular value of
$\alpha |_{\pi^{-1}(b)} \times M : = \alpha_b$.  \end{lemma}

In the following, we use that
\[
T_{u, \lambda, b} Q  =  \{(v, \eta, s) \in H^{1,0}(X) \times \mathbb{C} \times T_b B: \int_X u v dV_g  = 0, \; (L_b + \lambda) v + \eta u + D_2 \phi s = 0\}.
\]
and that  the  image of $D_2 \Phi$ is given by
\[
J = \mathrm{Image} D_2 \Phi_{(u, \lambda, b)}  = \{  \dot{\Delta} u:  \dot{\Delta} \; \text{ is a variation of } \; \Delta\; \text{ along a curve of metrics}\}.
\]

We often write $$v = \dot{u}, \eta = \dot{\lambda}, D_2(\Phi) s = \lambda \dot{\Delta} u, \;\; (\Delta + \lambda) \dot{u} + (\dot{\Delta} + \dot{\lambda}) u = 0.$$ Further, let $D_1 \alpha$ denote the derivative of $\alpha$ along $Q$. Then,
$$ D_{(u, \lambda, b)} \alpha(v, 0, c, 0) = v(x) = \dot{u}(x). $$

\subsection{Generic simplicity of eigenvalues }
In this section, we sketch the proof that the eigenvalues of $\Delta_G$ are of multiplicity one for generic bundle metrics; the details are similar
to those in \cite{JZ20} once the variations are calculated  and we do not repeat the overlapping arguments.
According to Theorem \ref{U},  we need to prove that the map $\Phi$ \eqref{CAPPHI}  is $C^k$ and has zero as a regular value.  By Proposition \ref{PROP}(b),
it suffices to show that  $W \in L^1(M)$ and
$W \in C^2(M - \{y\})$, the property $\int_M W(x) j(x) d\mu_x = 0$ for all $j \in J$
implies $W = 0$. In our problem,
$j  = \mathrm{im } D_2 \Phi$ are given by $j = \dot{\Delta}_G u$ where $(\Delta_G + \lambda) u =0$ for some $\lambda$, and where $\dot{\Delta}_G$ arises from variations
from  (i) general perturbations of  metrics on $M$, or
(ii) the space of connection 1-forms $\alpha$ on $\pi: X \to M$.
In the calculations we use that $\Delta_G =  \bigoplus_m L_m$ where $L_m$ operates on equivariant functions.
Not only is the spectrum of each $L_m $ \eqref{LmDEF}  simple, but the spectra of $L_m$ and $L_{m'}$
are disjoint sets if $m \not= m'$ for generic bundle metrics $G$. This last step is proved in \cite[Lemma 4.9]{JZ20}; the same proofs works in the
present higher dimensional setting and will not be repeated here.

\subsection{Generic base metrics}
The following Lemma proves that the criterion in Proposition \ref{PROP}

\begin{lemma}\label{JJLEM2}
Let $f$ be an eigenfunction of $L_m$ with the nonzero eigenvalue $\lambda \neq 0$. Let $g(t) $ be a $1$-parameter deformation of $g$ with
$g(0) = g$ and $b(t) g^{-1} = g(t)^{-1}$. Then the image of $\dot{\Delta} f$ is dense. \end{lemma}
\begin{proof}
We continue the perturbation calculation of Lemma \ref{JJLEM1}. In local coordinates, let
\begin{align*}
df+imf\eta &= a_i dx_i\\
dW+imW\eta &= w_i dx_i,
\end{align*}
and let
\[
d_g^* (\bar{W}(df+imf\eta)) = F.
\]
Then
\[
\langle \dot{L}_m f, W\rangle_{L^2(g)}=0
\]
is equivalent to
\begin{equation}\label{orth}
\int_M a_i b^{ij}g^{jk} w_k + \frac{F}{2}\mathrm{Tr}(b) dV_g = 0.
\end{equation}
We now assume for contradiction that $\dot{L}_m f$ does not have dense image, and $W$ is a nonzero section of $L^m$ which is orthogonal to $\dot{L}_m f$ under any variation. For \eqref{orth} to hold for all $b$, we must have
\[
a_i g^{jk}w_k + a_j g^{ik}w_k = 0
\]
for all fixed $i\neq j$, and
\[
a_i g^{ik} w_k = -\frac{F}{2}
\]
for all fixed $i$. In particular, we have
\[
(a_i dx_i)\wedge (g^{jk} w_k dx_j) = -\frac{F}{2} (dx_i \vee dx_i).
\]
Observe that the left hand side of the equation has rank at most $2$, while the right hand side has rank $n$ if $F\neq 0$, or $0$ if $F=0$. Because $n\geq 3$, we have
\[
F=0
\]
and
\[
(a_i dx_i)\wedge (g^{jk} w_k dx_j)=0
\]
on $M$. Because $g$ is clearly invertible, this implies that the support of
\[
df+imf\eta ~\text{ and } ~dW+imW\eta
\]
are disjoint, and
\[
d_g^* (\bar{W}(df+imf\eta))=0
\]
on $M$. Now we use the assumption that $f$ is an eigenfunction of $L_m$, i.e.,
\begin{equation}\label{eigenn}
\lambda f  = L_m f = d_g^* (df+imf\eta) - im (df+imf\eta, \eta)_{g^*}
\end{equation}
Then we have
\begin{align*}
0 &= d_g^* (\bar{W}(df+imf\eta))\\
& = -(df+imf\eta, dW)_{g^*} + \bar{W} d_g^*(df+imf\eta)\\
&= (df+imf\eta,imW\eta )_{g^*} + \bar{W} (\lambda f + im (df+imf\eta, \eta)_{g^*})\\
&= \lambda \bar{W} f,
\end{align*}
where we used $dW+imW\eta=0$ on the support of $df+imf\eta$ and \eqref{eigenn} in the third equation. This forces $W=0$ on an open set, and the Lemma follows from \cite[Proposition 2.10]{U}.
\end{proof}

\subsubsection{Generic connection $1$-forms}

\begin{lemma}\label{MULT1alpha} For generic connection 1-forms $\alpha$, the spectrum of $\Delta_G$ is of multiplicity one. \end{lemma}

\begin{proof} The variation of $\Delta_G$ induced by varying $\alpha$ is calculated  Section \ref{VARYALPHA}.
Assume for purposes of contradiction of Proposition \ref{PROP},  that there exists  $W \in L^1(M)$ and
$W \in C^2(M - \{y\})$ with  the property $\int_M W(x) \dot{\Delta}_G u d\mu_x = 0$  some $\Delta$-eigensection $u= f e_L^m$ and all variations. We need to prove that $W =0$.
The argument is almost identical to that in \cite{JZ20}, so we only briefly summarize it.

We prove this by representing equivariant functions on $X$ as sections of the associated bundle, and write sections relative to a local frame $e_L$.
    Exactly as in \cite{JZ20}, if the image is not dense, there exists $W = F e_L \not= 0$ so that (by Lemma \ref{VARYALPHA}),
    $$\begin{array}{l} \int_X \left(   - 2 i  G(df, \dot{\alpha})   + if d_g^* \dot{\alpha} + 2G(\dot{\alpha}, \alpha)  f\right) \bar{F} e^{-\psi} dV_g = 0 \\ \\ \iff
    \int_X \left(  ( - 2 i  G(df, \dot{\alpha})   + 2G(\dot{\alpha}, \alpha)  f) \bar{F}
     + i G(\dot{\alpha},  d (f \bar{F}))  \right) e^{-\psi} dV_g = 0, \end{array} $$
     for all $\dot{\alpha} \in \Omega^1(X).$ Here, we integrated $d_g^*$ by parts to remove it from $\dot{\alpha}$.
     As in \cite{JZ20}, this boils down to  the equation,      $$   ( - 2 i  df   + 2 \alpha  f) \bar{F}
     + i   d (f \bar{F})   = 0 \iff ( -i df + 2 \alpha f)\bar{F} + i f d\bar{F} = 0 \implies d \alpha =0. $$
     on a dense open set;  but a generic $C^k$ 1-form does not satisfy $d \alpha =0$.
 \end{proof}

\subsection{For generic bundle metrics the eigenfunctions have regular nodal sets}
To prove this statement we again use Theorem \ref{U}. To establish the hypothesis, we now use Lemma \ref{REG0LEM}.

Let
$$Q := \{(u, \lambda, b) \in C^k(X) \times \mathbb{R}_+ \times B: \Phi(u, \lambda, b) = 0\}, $$
and let
$$\alpha: Q \times M \to \C: \alpha(u, \lambda, b, x) = u(x). $$

To complete the proof of Theorem \ref{GENERIC1} it  suffices to recall the following Proposition, proved in \cite[Proposition 4.8]{JZ20}.
\begin{proposition}\label{SURJ2}
For each $m$,  $D_1 \alpha_m$ is surjective to $\C$.
We  need to show that for each $x \in M$,
\[
\alpha_m: Q \times \{x\} \to \C: \alpha(u, \lambda, g, x) = u(x)
\]
has $0 \in \C$ as a regular value, i.e., that
$$D_1 \alpha (\cdot, x): T_{u, \lambda, b}(Q) \to \C,\;\; D_1 \alpha(\cdot, x)_{(u, \lambda, g)} (\delta u(x), 0, c, 0) = \delta u(x)$$
is surjective to $\C$, where $D_1$ is the differential along $Q$ with $x \in M$ held fixed.\end{proposition}

The proof is very similar to that in \cite[Proposition 4.8]{JZ20}, so we only sketch enough of it to ensure that the argument used there
applies to our new types of deformations.

The proof is based on general features of the Green's function $$G_{m, \lambda}: [\ker (D_m + \lambda)]^{\perp} \to  [\ker (D_m + \lambda)]^{\perp},$$ for  i.e. the Schwartz kernel of the resolvent ($D_m(g) + \lambda
)^{-1}$ on the space where it is well-defined.

We  need to show that for each $x \in M$,
\[
\alpha_m: Q \times \{x\} \to \C: \alpha(u, \lambda, g, x) = u(x)
\]
has $0 \in \C$ as a regular value, i.e., that
$$D_1 \alpha (\cdot, x): T_{u, \lambda, b}(Q) \to \C,\;\; D_1 \alpha(\cdot, x)_{(u, \lambda, g)} (\delta u(x), 0, c, 0) = \delta u(x)$$
is surjective to $\C$, where $D_1$ is the differential along $Q$ with $x \in M$ held fixed.

As in the proof of \cite[Proposition 4.8]{JZ20}, $D_1 \alpha $
is  surjective to $\C$ unless for all $j \bot \ker (D_m(g) + \lambda)$,
either the real or imaginary parts of
$$G_{m, \lambda} (j)(x) =\int_M G_{m, \lambda}(x,y) j(y) dV(y) $$
vanishes  for every such $j$.
Since $j = [D_m(g) + \lambda] f$ where $\int f = 0$ we would  get  the absurd conclusion that
$$f(x) = 0, \;\; \forall f \bot \ker (D_m(g) + \lambda). $$
This is  impossible, concluding the proof.

\section{Proof of connectivity and smoothness of nodal sets and the $2$ nodal domain theorem for equivariant eigenfunctions with regular nodal set}
In this section, we prove the main result on connectivity. The proof is entirely different, and much simpler than, the proof of the analogous statement in
\cite{JZ20} and in particular does not use a local study of eigenfunctions as in \cite[Section 5]{JZ20}.
Recall from \eqref{REALHYP} that
\begin{equation*}
Z_{\mathrm{Re} \hat{s}} : = \{\mathrm{Re} \hat{s}(\zeta) = 0, \zeta \in \partial D_h^*\}.
\end{equation*}

\begin{proposition} \label{COVERPROP}
If $\hat{s} \in H_m(X)$ is an equivariant eigenfunction, then  natural projection $\pi: Z_{\mathrm{Re} \hat{s}} \to M$ is an $2m$-fold covering
map on the complement of the complex zero set $\{z~:~f_{m,j}(z) e_L^m=0\} \subset M$.
\end{proposition}

\begin{proof}
The proof is almost the same as in \cite{JZ20} and consists of a sequence of Lemmas.  In the notation of Section \ref{EQEIGSECT},
we  denote by $\zcal_{f_{m,j}}$ the  zero set of the eigensection $f_{m,j} e_L^m$ on $M$:
\begin{equation*}
\zcal_{f_{m,j}} = \{z \in M: f_{m,j}(z) = 0\}.
\end{equation*}
We denote by $\zcal_{\phi_{m,j}}$  the nodal set of the (complex-valued) equivariant eigenfunction $\phi_{m,j}$ on $X$.  Under the natural projection $\pi: X \to M$,
\begin{equation*}
\zcal_{\phi_{m,j}}  = \pi^{-1} \zcal_{f_{m,j}}.
\end{equation*}

\end{proof}

\begin{lemma} \label{REG}
For $z \in M$ such that $f_{m, j} (z)\neq 0$, there exist
 $2 m$ distinct  solutions $v$ of  $\mathrm{Re} f_{m,j} e_L^m(v)= 0$ with  $v \in L^*_z X$.
 \end{lemma}

  \begin{proof} We trivialize $L_z^*$ using the dual frame $e_L^* \sim \C$ and use polar coordinates $(r, \phi)$ on $\C $.  Since the equations
  are homogeneous, we set $r=1$ and identify
$v = (\cos \theta, \sin \theta)$.
In the notation of \eqref{ubab}, the equation for a nodal point is,
$$ \left(a_{m,j} \mathfrak{c}_m -  b_{m,j} \mathfrak{s}_m\right)(\cos \theta, \sin \theta) = 0. $$
Here $\mathfrak{c}_m = \mathrm{Re} (\cos \theta + i \sin \theta)^m = \cos m \theta,$ and the equation is
\[
a_{m,j}(z) \cos m \theta -  b_{m,j}(z) \sin m \theta = 0 \iff \tan m \theta = \frac{a_{m,j}}{b_{m,j}},
\]
where $a_{m,j}, b_{m,j} \in \mathbb{R}$ and where we assume with no loss of generality that $b_{m,j} \not= 0$. For $0\leq \theta <2\pi$, we have $0\leq m\theta <2m\pi$, and so there are exactly $2m$ choices of $\theta$.
\end{proof}

The following is an immediate consequence of Lemma \ref{REG}:

\begin{lemma}
If $0$ is a regular value, then the nodal set  $\ncal_{m,j} \subset X$ of $\mathrm{Re} \phi_{m,j}$ is a singular $2m$-fold cover of $M$ with blow-down singularities over  points where $f_{m,j}(z) e_L^m = 0$.
\end{lemma}

Indeed, the $2m$ zeros of $\mathrm{Re} \omega_{m,j}(v) = 0$ in $S_z X$ give $2m$ points on the fiber $\pi^{-1}(z)$ in $P_h$. Since
locally there exist $2m$ smooth determinations of the zeros, the nodal set is a covering map away from the singular points.
This completes the proof of Proposition \ref{COVERPROP}.

We prove the connectivity statement of Theorem \ref{MAIN}.

\begin{proposition} \label{CONNECTEDPROP} 
If $\exists z_0$ such that $f_{m,j}(z_0) e_L^m=0$, then the  nodal set of $\mathrm{Re} \phi_{m, j}$ is connected. \end{proposition}
\begin{proof}
Let $\emptyset \neq \Sigma \subset M$ be the zero set of $f_{m, j} e_L^m$. By Proposition \ref{COVERPROP}, $ \ncal_{m, j}   \backslash (\ncal_{m, j}  \cap \pi^{-1} (\Sigma)) \to M \backslash \Sigma$
is an $2m$-sheeted cover. Because $M$ is connected, any point on $M$ is connected to a point in $\Sigma$. Therefore any point on $N_{m,j}$ is connected to a lift of a point in $\Sigma$. Because the lift of a point in $\Sigma$ is a circle, this proves the connectivity of any two points in the lift of a point in $M\backslash \Sigma$. Now, again by the connectivity of $M$, for any given $p,q\in M$, we see that any lift of a point $p$ is connected to some lift of point $q$. This concludes that any lift of $p$ is connected to any lift of $q$, which shows that $\ncal_{m,j}$ is connected.
\end{proof}

We now prove that the nodal set is smooth.

\begin{lemma}
If $0$ is a regular value of $f_{m,j}$, then $\ncal_{m,j} \subset X$ is a smooth submanifold of $X$.
\end{lemma}
\begin{proof}
Assume that $f_{m,j}(z_0)=0$ is a regular zero. Then
\[
Df_{m,j}(z_0) : T_{z_0} (M) \to \mathbb{C}
\]
is a surjection. Now for any $0\leq \theta <2\pi$,
\[
Du_{m,j}(z_0,\theta) : T_{z_0,\theta} (X) \to \mathbb{C}
\]
restricted to the vectors tangent to $M$ is given by
\begin{equation}\label{eqeq11}
\cos m\theta Da_{m,j} (z_0) -\sin m\theta Db_{m,j} (z_0),
\end{equation}
which is the same as multiplying $Df_{m,j}(z_0)$ by $e^{im\theta}$ and then taking the real part. Since the image of $Df_{m,j}(z_0)$ is entire $\mathbb{C}$, it is clear that \eqref{eqeq11} surjects onto $\mathbb{R}$. This implies that $Du_{m,j}(z_0,\theta)$ is surjective for any $\theta$, and therefore $\ncal_{m,j} \subset S X$ is a smooth submanifold, by Proposition \ref{COVERPROP}.
\end{proof}

Finally, we complete the proof of Theorem \ref{MAIN}. Note that the existence of $z_0$ such that $f_{m,j}(z_0)=0$ follows from the nontriviality of the principal bundle.

\begin{proposition} If $0$ is a regular value of $\phi_{m, j}$ and there exists $z_0$ such that $f_{m,j}(z_0)=0$, then $SX \backslash \ncal_{m,j}$ has exactly two connected components. \end{proposition}

\begin{proof}
By Proposition \ref{CONNECTEDPROP}, if $0$ is a regular value then  $ \ncal_{m,j}$ is a connected, embedded hypersurface. To prove
that it separates $X$ into two connected components, note first that any connected embedded hypersurface of $X$  may either separate $X$ into two components, or $X \backslash \ncal_{m,j}$  is connected. Since $\mathrm{Re}\phi_{m,j}$ is a non-constant eigenfunction of the Laplacian, it integrates to $0$, and so there are at least two nonempty nodal domains one in $\{\mathrm{Re}\phi_{m,j}>0\}$ and the other one in $\{\mathrm{Re}\phi_{m,j}<0\}$. Therefore $\ncal_{m,j}$ separates $X$ into exactly two connected components.
\end{proof}

\section{\label{EVENDIM} Even dimensional spheres with effective  $S^1$ actions: Proof of Lemma \ref{BADSPHERE}}

The results above do not apply to even dimensional spheres, or any manifold with an effective $S^1$ action whose Euler characteristic is non-zero, since then $Z$ has zeros and $S^1$ cannot act freely. An obvious question is, what can be said of nodal sets of real parts of equivariant eigenfunctions in this case? For a  non-free  $S^1$ actions by isometries of a Riemannian manifold $(X, G)$,
the  lapse function  $|Z|$ is never constant  in this case and of course equals zero at the fixed points.  It is still the case that
 $[\Delta_G, D_Z] = 0$ and there still exists an orthonormal basis $\{\phi_{m, j}\}$ of joint (equivariant) eigenfunctions. There
 are many possible types of examples which depend, for instance, on the nature of the fixed point set of $S^1$.

 The simple reason why nodal sets are singular in even dimensions is the existence of fixed points of the $S^1$ action and  that,
 at the fixed points $p$, every equivariant eigenfunction with $m \not=0$ must
vanish, since $ \phi_{m, j}(p)  = \phi_{m,j}(r_{\theta} p) = e^{i m \theta} \phi_{m, j}(p) $.  Moreover, each fixed point is a critical point since
$d  \phi_{m, j}(p)  = d \phi_{m,j}( p)  \circ D r_{\theta} = e^{i m \theta} d \phi_{m, j}(p) $ where $D_p r_{\theta}: T_p X \to T_p X.$
Thus, each fixed point is a singular point of $\phi_{m,j}$ (a critical nodal point), i.e.  $0$ is not a regular value of any equivariant eigenfunction or of
$\mathrm{Re} \phi_{m, j}$  and obviously Theorem \ref{MAIN} does not hold for any bundle metric in this case.

We now consider in more detail  the case of even dimensional spheres $\Ss^{2n} \subset \mathbb{R}^{2n+1}$ and assume that the $S^1$ action is a one-parameter subgroup of $SO(2n+1)$
and that it has precisely two fixed points on $\Ss^{2n}$, which we assume to be $\pm e_{2n +1}$ where $e_j$ is the standard basis. We denote $N$th
degree equivariant spherical harmonics by $\phi_N^m$. These form a $\sim N^{d-2} $ dimensional subspace and to uniquely specify the harmonic
we could use an orthonormal basis of joint eigenfunctions of the Laplacian and a Cartan subgroup. But we stick to this simple but ambiguous notation,
since the the claims hold for any element of the subspace.

Let $G= SO(2n+1)$ and let $G_p$ denote the isotropy group of $p \in \Ss^{2n}$. If  $p = e_{2n+1}$, then
the $G_p$ induces a derived action on  $T_{e_{2n+1}} \Ss^{2n}$. We define $S^1$ to be a one-parameter subgroup of $G_p$ which has precisely
two fixed points $p, -p$ (in the sense that $S^1_p = S^1$ or $Z (p) = 0$).  For instance, we can let $S^1$ be the direct sum of $2 \times 2 $ blocks $\begin{pmatrix} \cos \theta & \sin \theta \\ & \\
- \sin \theta & \cos \theta \end{pmatrix}$ plus a $1\times 1$ block with entry $1$ corresponding to $e_{2n+1}$. The $S^1$ action preserves $\mathbb{R}^{2n} : = (\mathbb{R} e_{2n+1})^{\perp}$ and therefore the equatorial subsphere $\Ss^{2n-1} \subset \mathbb{R}^{2n}$. The quotient $ \Ss^{2n+1}/S^1$ is not a smooth manifold due to the fixed points, but the standard metric is bundle-like with respect to the
natural projection  $\pi: \Ss^{2n+1} \to \Ss^{2n +1}/S^1$ since the map $\pi: \Ss^{2n +1} \backslash \{p, -p\} \to  (\Ss^{2n +1}  \backslash \{p, -p\}) /S^1$ is
a smooth projection.  It also preserves all ``latitude spheres'' defined as level sets of $x_{2n +1}: \Ss^{2n} \to \mathbb{R}$
and acts freely on each. It follows that  $ \Ss^{2n} \backslash \{p, -p\} \simeq (-\frac{\pi}{2}, \frac{\pi}{2}) \times \Ss^{2n-1}$ and the quotient by the $S^1$ action
is  given by the diffeomorphism   $ \Ss^{2n} \backslash \{p, -p\} /S^1 \simeq (-\frac{\pi}{2}, \frac{\pi}{2}) \times (\Ss^{2n-1} /S^1)$. We may therefore consider
equivariant eigenfunctions in a similar spirit to the case of odd-dimensional spheres. A key difference is that $Z$ is not of constant norm, so the bundle
metric necessarily has a non-constant lapse function $N$ depending on (and only on) $x_{2n+1}$; it is given by $N(x_{2n +1}) = \sin r(x_{2n+1})$ where $r $ is the distance
of the latitude sphere of height $x_{2n+1}$ to $p$.

The free $S^1$ action on    $ \Ss^{2n} \backslash \{p, -p\} $ determines  complex line bundle $L^m \to (-\frac{\pi}{2}, \frac{\pi}{2}) \times (\Ss^{2n-1} /S^1)$
associated to characters $\chi_m$ and as in the bundle case, the equivariant eigenfunctions of $\Delta_{\Ss^{2n}}$ correspond to complex eigensections of this
line bundle for the induced operators $L_m$ \eqref{LmDEF}. Note that $(\Ss^{2n-1} /S^1) \simeq \CP^{n-1}$.

We now prove Lemma \ref{BADSPHERE}.

\begin{proof}

For $n > 1$, on  $\Ss^{2n}$ one  has a map $q: \Ss^{2m} \backslash \{\pm e_{2n+1}\} \to \Ss^{2n-1}$ where $\Ss^{2n-1}$ is the equatorial sphere, obtained by following
the orthogonal geodesics to $\Ss^{2n-1}$ to the poles. The second map is  again $x_{2n +1}: \Ss^{2n} \to (-1,1)$. Together we have a 1-parameter family of
latitude spheres and a $2n-1$-parameter family of orthogonal longitude lines. There is a third map $ \Ss^{2n} \backslash \{p, -p\} /S^1 \simeq (-\frac{\pi}{2}, \frac{\pi}{2}) \times (\Ss^{2n-1} /S^1)$ to the orbit space. Since $x_{2n+1}$ is constant on $S^1$ orbits, the third map and second map coincide when $n = 2$.

A unitary section $u:  \Ss^{2n}  \backslash \{p, -p\} /S^1  \to \Ss^{2n}$ is an inverse of $ \Ss^{2n} \to \Ss^{2n}  \backslash \{p, -p\} /S^1$ giving a cross section
to the $S^1$ action. When $n=2$ a section is given by a meridian line. This would be a section of the $SO(2n -1)$ action stabilizing the poles. Now it is given by
a family $(2n-2)$-dimensional family of   meridian lines which is the flowout of the transversal to the $S^1$ orbit in $S^*_p \Ss^{2n}$. An equivariant function
has the form $\phi_N^m(u e^{i \theta}) = e^{im \theta} \phi_N^m(u)$. The real part is $ \mathrm{Re} \phi_N^m =  u_N^m(u) \cos m \theta $, where we assume that $u_N^m(u)$
is real valued  If we consider the $2m$ zeros $\theta_{m, j}$ of $\cos m \theta$, we get the union $\bigcup_{j=1}^{2m} \{\theta_{m, j}\} \times  \Ss^{2n}  \backslash \{p, -p\} /S^1$,  disjoint union of $2m$ hypersurfaces. When $m \not=0$ the submanifolds meet at the poles.  There also exist $S^1$ invariant transverse components
of the nodal set coming from the factor $\phi_N^m(u)$.

In the notation of Section \ref{EQEIGSECT}, the $S^1$ fibers of the zero set of $\phi_N^m$  over the complex zero set $\{f^m_{N} = 0\} \subset  \Ss^{2n}  \backslash \{p, -p\} /S^1$ of the sections of $L^m$ intersect the  set  $\bigcup_{j=1}^{2m} \{\theta_{m, j}\} \times  \Ss^{2n}  \backslash \{p, -p\} /S^1$.
For $n \geq 2,$ $f_m^N$ is complex valued, and this set is  of real codimension $3$.

The  nodal set of $\phi_N^m$ is $S^1 \times \{f^m_{N} = 0\} $ and it obviously intersects
$\bigcup_{j=1}^{2m} \{\theta_{m, j}\} \times  \Ss^{2n}  \backslash \{p, -p\} /S^1$ in the set
$\bigcup_{j=1}^{2m} \{\theta_{m, j}\} \times   \{f^m_{N} = 0\} $.

This completes the proof of Lemma  \ref{BADSPHERE}.

 \end{proof}

The most familiar case $n=1$ is very special and we pause to contrast it with the case $n > 1$. In this case, the quotient is simply $(-\frac{\pi}{2}, \frac{\pi}{2})$.
The equivariant eigenfunctions form the standard basis $Y_N^m(\theta, \phi) = \sqrt{2N +1}  e^{i m \theta} P_N^m(\cos \phi)$. Here, $P_N^m$ are the standard
associated Legendre polynomials and \eqref{LmDEF} is the $m$th Legendre operator. The real part of $Y_N^m$ is $\sqrt{2N +1}  P_N^m(\cos \phi) \cos m \theta$.
Its nodal set is the union of the nodal sets $\{P_N^m = 0\} \cup \{\cos m \theta = 0\}$, and it has a singular set of $m N$ points given by the   Cartesian product of the zero set of $\cos m \phi$ and the zero set of $P_N^m(\cos \phi)$.  The complex line bundle $L_m$ is the product bundle $(-\frac{\pi}{2}, \frac{\pi}{2}) \times \C$. In the
constant frame, the eigensections are all real-valued and are the Legendre factors $  P_N^m(\cos \phi) $. The inverse image of their zero sets thus has real codimension
one rather than $2$. Over the complement of $\{P_N^m = 0\}$, the zero set of  $\mathrm{Re} Y^m_N$ is still an  $m$-fold cover of the base, indeed consists of the $m$ zeros
of $\cos m \theta$ on each latitude circle and the real  zero set consists of $m$ longitude geodesics through these zeros. The special features are that the complex
zero set of $Y_N^m$ is a union of $N$ latitude  circles rather than being of real codimension $2$ and that the $m$ longitude circles all meet at the poles. The nodal
set is connected but has $N m$ singular points and $N m$ nodal domains.


\bibliography{bibfile}
\bibliographystyle{alpha}
\end{document}